\documentclass[12pt]{amsart}

\usepackage{amssymb}
\usepackage{amsthm}

\newtheorem{theorem}{Theorem}
 \newtheorem{thm}[theorem]{Theorem}
 \newtheorem{lem}[theorem]{Lemma}
 \newtheorem{lemma}[theorem]{Lemma}
 
 \newtheorem{cor}[theorem]{Corollary}
 \newtheorem{prop}[theorem]{Proposition}

 \newtheorem{rem}[theorem]{Remark}
 \newtheorem{defn}[theorem]{Definition}

 \def\Bbb{\mathbb}
 \def\Z{\Bbb{Z}}
 \def\bZ{\Bbb{Z}}
 
 \def\Q{\Bbb{Q}}

\newcommand{\ba}{\begin{array}}
\newcommand{\ea}{\end{array}}
\newcommand{\be}{\begin{enumerate}}
\newcommand{\ee}{\end{enumerate}}
\newcommand{\bi}{\begin{itemize}}
\newcommand{\ei}{\end{itemize}}

\newcommand{\bc}{\begin{center}}
\newcommand{\ec}{\end{center}}
\newcommand{\bt}{\begin{tabular}}
\newcommand{\et}{\end{tabular}}

 \def\isom{\cong}
 \def\onto{{\kern3pt\to\kern-8pt\to\kern3pt}}
 
 \def\<{\langle}
 \def\>{\rangle}
 \def\|{{\ |\ }}

 \def\rank{\mbox{rank}}
 \def\ov{\overline}

\title[Decomposability of nilpotent groups]
{Decomposability of finitely generated torsion-free nilpotent groups}
\date{\today}

\author[Baumslag]{Gilbert Baumslag}\address{Gilbert Baumslag\\
Department of Mathematics and Computer Science\\
City College of New York\\
Convent Ave. and 138th Street\\
New York, N.Y. 10031\\USA}

\author[Miller]{Charles~F. Miller~III }
\address{ Charles~F. Miller~III\\
Department of Mathematics and Statistics\\
University of Melbourne\\
Melbourne 3010 \\Australia }
\email{c.miller@ms.unimelb.edu.au }

\author[Ostheimer]{Gretchen~Ostheimer}
\address{ Gretchen~Ostheimer\\
Department of Computer Science\\
211 Adams Hall\\
Hofstra University \\
Hempstead NY 11549 \\USA}
\email{gretchen.ostheimer@hofstra.edu}

\begin{document}

\begin{abstract}
We describe an algorithm for deciding whether or not 
a given finitely generated torsion-free nilpotent group 
is decomposable as the direct product of nontrivial subgroups.
\end{abstract}



\maketitle

\date{\today}

\section{Introduction}

Finitely generated nilpotent groups seem tractable from some points of view.  
Such a nilpotent group $G$ is finitely presented, and the elements of finite order form 
a finite normal subgroup $T$ with torsion-free quotient $G/T$.  Moreover many algorithmic
problems have positive solutions for finitely generated nilpotent groups.
For example, the word and conjugacy problems can be solved in a number of ways.  
Perhaps most remarkably,  Grunewald and Segal \cite{GS} have solved the isomorphism problem for finitely generated nilpotent groups.   

In this paper we address a still open decidability question for these groups, raised by Baumslag in \cite{Gilbert:open}: determine 
whether a nilpotent group given by a finite presentation has a nontrivial direct product
decomposition.  We show that such an algorithm exists for the subclass of torsion-free
finitely generated nilpotent groups.

Two common algorithmic approaches are (1) using residual properties and (2) using 
a polycyclic series inductively.  So the conjugacy problem for nilpotent groups can be solved 
by showing such groups are conjugacy separable, that is, non-conjugate elements remain non-conjugate in some finite quotient.  Enumeration arguments then provide an algorithm to determine conjugacy.    The second approach also gives algorithms solving a wide variety of problems for nilpotent and polycyclic groups (\cite{BCRS},\cite{BCM1},\cite{BCM2}) often using an effective version of the Hilbert basis theorem.   

There are some known difficulties with nilpotent groups.  Remeslennikov \cite{Remeslennikov:finiteQuot} constructs non-isomorphic finitely presented nilpotent groups which 
have the same collection of finite quotient groups.   Perhaps more ominously, Remesennikov
\cite{Remeslennikov:epi} shows that while one can determine whether one nilpotent group embeds in another, there is no algorithm to determine whether one is a quotient of another.  He shows 
Hilbert's tenth problem is reducible to this epimorphism problem. 

Moreover, the Remak-Krull-Schmidt theorem fails for finitely generated nilpotent groups,
because direct product decompositions, when they do exist, are far from unique: in \cite{Baumslag:nilpDecomp}, Baumslag shows that for any pair of integers $m,n > 1$, it is possible to construct a single torsion-free nilpotent group with two different direct product decompositions, one with $m$ indecomposable factors, the other with $n$ indecomposable factors, where no factor in the first decomposition is isomorphic to any factor of the second decomposition. 
An analysis of Baumslag's non-uniqueness examples led  us to the following 
theorem. 

\newtheorem*{thm:a}{Theorem \ref{thm:a}}
\begin{thm:a}  There is an algorithm to determine of an arbitrary finite presentation of 
a torsion-free nilpotent group $G$ whether or not $G$ has an abelian direct factor.
If so, the algorithm expresses $G$ as $G \isom G_1 \times \Z^n$ where $G_1$
has no nontrivial abelian direct factor.
\end{thm:a}

In Section \ref{examples} we illustrate how 
the existence of abelian direct factors can be a source of non-uniqueness.
The algorithm of Theorem \ref{thm:a} combines some elementary considerations
with several known algorithms for presenting subgroups of abelian and nilpotent groups.
Making progress in the absence of abelian direct factors involves more elaborate
methods.  We rely on properties of the rational closure (Malcev completion) of torsion-free nilpotent groups 
and use uniqueness of decomposition results for rational Lie algebras.
Our result is the following:

\newtheorem*{thm:b}{Theorem \ref{thm:b}}
\begin{thm:b} There is an algorithm to determine of an arbitrary finite presentation of 
a torsion-free nilpotent group $G$ without abelian direct factors, 
whether or not $G$ has a nontrivial direct decomposition.
If so, the algorithm expresses $G$ as $G \isom G_1 \times\ldots  \times G_n$ where 
each $G_i$ is directly indecomposable.
\end{thm:b}

Our paper is structured as follows.
In Section \ref{background} we present some background material about the rational closures of finitely generated torsion-free nilpotent groups. We believe that these results are probably well-known, but since we have not been able to find references, we include proofs here. 
In Section \ref{abelianSection} we prove Theorem \ref{thm:a}.
In Section \ref{nonabelianSection} we present some structural theorems that describe the relationship between the myriad decompositions of a torsion-free nilpotent group and the more constrained decompositions of its rational closure and we use these to prove Theorem \ref{thm:b}.
In Section \ref{examples} we use the examples from \cite{Baumslag:nilpDecomp} to illustrate our algorithm. 

We leave three obvious questions unanswered. First, can our result be extended to include groups with torsion? Second, is the algorithm presented here practical; that is, is it possible to implement this algorithm (or a variant of it) in such a way that the algorithm can be used to determine the decomposability (and also to find a decomposition) in reasonably complex examples?
Third, if a finitely generated torsion-free nilpotent group does not have any nontrivial abelian factors, 
is its decomposition as a direct product of directly indecomposable groups unique up to isomorphism?

{\it  In memoriam Gilbert Baumslag:}  This work  results from discussions among the authors at various times, particularly during July and August of 2014.  In September of that year Gilbert was diagnosed with incurable pancreatic cancer and he died on 20 October.  His passing was of great sadness to us and to his many friends and colleagues.   Gilbert's contributions to group theory
were vast, he enjoyed sharing ideas and collaborated widely,  and he gave assistance generously
to students and younger colleagues.  We miss him greatly.

\section{Background material about the rational closure}
\label{background}

In this section we gather together results about the rational closures of finitely generated torsion-free nilpotent groups. 
We suspect that  all of the results presented here are well-known. For those results for which we have been unable to find references, we include our own proofs. 

For every finitely generated torsion-free nilpotent group $G$, there exists a torsion-free nilpotent group
$\overline{G}$  satisfying the following properties:
\begin{itemize}
\item $G$ embeds in $\overline{G}$;
\item for all $h \in \overline{G}$ and for all positive integers $\alpha$, there exists a unique
element $k \in \overline{G}$ such that $k^\alpha = h$;
\item for all $h \in \overline{G}$ there exists a positive integer $\alpha$ such that $h^\alpha \in G$.
\end{itemize}
$\overline{G}$ is unique up to isomorphism and it is called the {\em rational closure} of $G$
(see Chapter 6 in \cite{Segal:book}).

In order to understand the relationship between the direct product decompositions 
of $G$ and those of $\overline{G}$, we need two straightforward results:
first, a direct decomposition of $G$ gives rise to a direct decomposition of $\overline{G}$;
second, the well-known theorem regarding the uniqueness of direct sum decompositions of Lie algebras
can be reframed to give a useful description of the uniqueness of the direct product decompositions of $\overline{G}$.
There are a number of ways to approach these proofs. 
Here we  choose to exploit the fact that our groups can be represented by unitriangular matrices with 
integer entries and that in this context we can use the logarithm map to embed our groups in 
a finite dimensional Lie algebra. (This approach is described in \cite{Segal:book}, for example.)
The reader who is willing to accept Proposition \ref{easyDirection}
and Proposition \ref{rationalUnique} below can skip to Section \ref{abelianSection}.

For ring $S = \Z, \Q$ and for $m = 0, 1 \in S$, we let 
$Tr_m(r, S)$ denote the set of $r \times r$ upper-triangular matrices with entries in $S$ and $m$'s on the diagonal.
Every finitely generated torsion-free nilpotent group
can be embedded in the group $Tr_1(r, \Z)$
for a suitably chosen $r$ (see, for example, Chapter $5$ in \cite{Segal:book}). 
Some of the proofs here will be easy 
using such a matrix representation, so we will assume that our given group is a subgroup
of $Tr_1(r, \Z)$ whenever it is convenient to do so. 

Recall that for $x \in Tr_1(r, \Q)$, $\log(x)$ is defined by 
$\log(x) = u - \frac{1}{2}u^2 + \frac{1}{3}u^3 - \cdots$,
where $u = x-1$. 
For $u \in Tr_0(r, \Q)$, $\exp(u)$ is defined by
$\exp(u) = 1 + u + \frac{1}{2!}u^2 + \frac{1}{3!}u^3 + \cdots$. 
In both cases, since $u^r = 0$, the indicated sum is finite. 

The following standard properties of $\log$ and $\exp$ can
be found in \cite{Segal:book}, for example. 
\begin{rem}
\label{logExpInv}
For all $x \in Tr_1(r, \Q)$ and all $u \in Tr_0(r, \Q)$,
$\exp(\log(x)) = x$ and $\log(\exp(u)) = u$.
\end{rem}

\begin{rem}
\label{logOfPower}
 For all $x \in Tr_1(r, \Q)$ and all non-negative integers $n$, $\log(x^n) = n\log(x)$ .
 \end{rem}
 
The $\log$ and $\exp$ maps can be used to construct $\overline{G}$ as follows (see \cite{Segal:book} for example). 
\begin{prop}
Let $G$ be a subgroup of $Tr_1(r, \Z)$. 
Let $L$ be the vector space of $Tr_0(r, \Q)$ generated by $\{ \log(g) \; | \; g \in G\}$.
Let $H = \exp(L)$. 
Then $H$ is the rational closure of $G$. 
\end{prop}

\begin{prop}
\label{logExpComm}
Let $x_1,x_2 \in Tr_1(r, \Q)$, $u_1, u_2 \in Tr_0(r, \Q)$. 
Then $x_1$ and $x_2$ commute if and only if $\log(x_1)$ and $log(x_2)$ commute.
Likewise $u_1$ and $u_2$ commute if and only if $\exp(u_1)$ and $\exp(u_2)$ commute.
\end{prop}
\begin{proof}
Let $x_1$ and $x_2$ be commuting matrices in $Tr_1(r, \Q)$.
Let $u_i = x_i-1$. Then $u_1$ and $u_2$ commute. Thus, from the definition of $\log$, 
we see that  $\log{x_1}$ and $\log{x_2}$ also commute. 
From this we also see that if $\exp(u_1)$ and $\exp(u_2)$ commute,
then by Remark \ref{logExpInv} so do $u_1 = \log(\exp(u_1))$ and $u_2 = \log(\exp(u_2))$.

Now let $u_1$ and $u_2$ be commuting matrices in $Tr_0(r, \Q)$. 
By the definition of  $\exp$, $\exp(u_1)$ and $\exp(u_2)$ also commute. 
From this we also see that if $\log{x_1}$ and $\log{x_2}$ commute,
then by Remark \ref{logExpInv} so do $x_1 = \exp(\log(x_1))$ and $x_2 = \exp(\log(x_2))$.
\end{proof}

In $Tr_0(r, \Q)$ we will denote by $(u,v)$ the Lie bracket $uv-vu$.
\begin{prop}
\label{easyDirection}
Let $H$ be a finitely generated torsion-free nilpotent group. 
If $H = H_1 \times H_2$, 
then $\overline{H} = \overline{H_1} \times \overline{H_2}$.
\end{prop}
\begin{proof}
We may assume that $H$ is a subgroup of $Tr_1(r, \Z)$.
We first show that $\overline{H_1}$ and $\overline{H_2}$ commute.
Let $k_1 \in \overline{H_1}$ and $k_2 \in \overline{H_2}$. 
There exist positive integers $m_1$ and $m_2$ such that 
$k_1^{m_1} \in H_1$ and $k_2^{m_2} \in H_2$.
Therefore, by Remark \ref{logOfPower},
\begin{eqnarray*}
0 & = & (\log(k_1^{m_1}), \log(k_2^{m_2})) \\
& = & (m_1 \log(k_1), m_2 \log(k_2)) \\
& =&  m_1 m_2 (\log(k_1), \log(k_2)).
\end{eqnarray*}
Therefore $(\log(k_1), \log(k_2)) = 0$, 
and hence by Proposition \ref{logExpComm}, 
$k_1$ and $k_2$ also commute, as desired.

It is easy to see that $\overline{H_1} \cap \overline{H_2} = 1$.
If $h \in  \overline{H_1} \cap \overline{H_2} $, then there exist positive integers $m_1$ 
and $m_2$ such that $h^{m_1} \in H_1$ and $h^{m_2} \in H_2$.
Thus $h^{m_1 m_2} \in H_1 \cap H_2 = 1$.
Since $Tr_1(n, \Q)$ is torsion-free, $h = 1$.

Finally we show that $\overline{H} \subseteq \overline{H_1} \times \overline{H_2}$. 
Suppose that $h \in \overline{H}$.
Then there exists a positive integer $m$ and elements $h_1 \in H_1, h_2 \in H_2$
such that $h^m = h_1 h_2$. Let $r_1$ and $r_2$ be the $m$'th roots of $h_1$ and $h_2$ 
respectively. 
Since $\overline{H_1}$ and $\overline{H_2}$ commute, 
$$(r_1 r_2)^m = r_1^m r_2^m = h_1 h_2 = h^m.$$
Since roots are unique in $Tr_1(r, \Q)$,  $h = r_1 r_2$.
\end{proof}

The upper central series plays a special role in the relationship between
a finitely generated torsion-free group and its rational closure, as the following
well-known theorem asserts (see \cite{Kurosh:book2} p. 257 for a proof and 
a discussion of the history of this result).
\begin{thm}
\label{upperCentral}
Let $G$ be a finitely generated torsion-free nilpotent group. 
Let $\Gamma_i(G)$ be the $i$'th term in the upper central series of $G$.
Then $\Gamma_i(\overline{G}) = \overline{\Gamma_i(G)}$ 
and $\Gamma_i(G) = \Gamma_i(\overline{G}) \cap G$. 
\end{thm}

We will now describe the strong sense in which decompositions of rational nilpotent groups are unique. 
We begin with a classical result about the uniqueness of decompositions in Lie algebras.
Let $L$ be a Lie algebra, and let $(i,j)$ denote the Lie bracket of two elements $i$ and $j$ in $L$. 
Recall that a subspace $J$ of  $L$ is an {\em ideal} if 
for all $j \in J$ and all $l \in L$, $(j, l) \in J$ and $(l, j) \in J$,
and such an ideal is {\em indecomposable} if it cannot be written as the direct
sum of two nontrivial ideals. 
$L$ is Artinian (resp. Noetherian) it it satisfies the descending (resp. ascending) chain 
condition on ideals.

The following is proved in \cite{FisherGrayHydon}.
\begin{thm}
\label{uniqueLie}
Let $L$ be a Lie algebra that is both Artinian and Noetherian.
Suppose also that $L$ has two decompositions as a direct sum of nontrivial indecomposable ideals:
\begin{eqnarray*}
L & = & M_1 \oplus M_2 \oplus \cdots \oplus M_r \\
& = & N_1 \oplus N_2 \oplus \cdots \oplus N_s
\end{eqnarray*}
Let $\pi_i$ be the projection of $L$ onto $M_i$,
and let $\psi_i$ be the projection of $L$ onto $N_i$.
Then $r = s$ and the summands can be reordered such that the following hold
for all $1 \leq k \leq r$:
\begin{itemize}
\item $\pi_k(N_k) = M_k$ and the restriction of $\pi_k$ to $N_k$ is an isomorphism from $N_k$ to $M_k$
whose inverse is the restriction of $\psi_k$ to $M_k$;
\item  
\begin{eqnarray*}
L & = & M_1 \oplus M_2 \oplus \cdots \oplus
M_{k} \oplus N_{k+1} \oplus \cdots \oplus N_r.
\end{eqnarray*}
\end{itemize}
\end{thm}

We will need a slight reformulation:
\begin{cor}
\label{uniqueLieCor}
Let $L$ be a Lie algebra satisfying the hypotheses of Theorem \ref{uniqueLie}.
Then for any $k$ such that $1 \leq k \leq r$, 
\begin{eqnarray*}
L & = & M_1 \oplus M_2 \oplus \cdots \oplus M_{k-1} \oplus N_{k} \oplus M_{k+1} \cdots \oplus M_r.
\end{eqnarray*}
\end{cor}
\begin{proof}
To see that
\begin{eqnarray*}
L & = & M_1 + M_2 + \cdots + M_{k-1} + N_{k} + M_{k+1} \cdots + M_r,
\end{eqnarray*}
we need to show that $M_k \subseteq M_1 + M_2 + \cdots + M_{k-1} + N_{k} + M_{k+1} \cdots + M_r$.
Let $m \in M_k$, and let $n_k = \psi_k(m) \in N_k$. Then $\pi_k(m - n_k) = 0$, so 
$m- n_k \in M_1 + \cdots + M_{k-1} + M_{k+1} + \cdots + M_r$.
Thus $m \in M_1 + M_2 + \cdots + M_{k-1} + N_{k} + M_{k+1} \cdots + M_r$.

It is clear from the statement of Theorem \ref{uniqueLie} that $N_k$ commutes with and is disjoint from $M_1 \oplus M_2 \oplus \cdots \oplus M_{k-1}$,
and by reversing the roles of the $M_i$'s and $N_i$'s in Theorem \ref{uniqueLie},
it is clear also that $N_k$ commutes with and is disjoint from $M_{k+1} \oplus  \cdots \oplus M_r$.
\end{proof}

The $\log$ and $\exp$ maps satisfy the following well-known properties. 
\begin{prop}
\label{logExpProps}
Let $H_1, H_2$ be subsets of $Tr_1(r, \Q)$ and let $M_1, M_2$ be subsets of $Tr_0(r, \Q)$
such that $M_i = \log(H_i)$.
Then 
\begin{enumerate}
\item $H_1$ and $H_2$ commute if and only if $M_1$ and $M_2$ commute;
\item $H_1 \cap H_2 = 1$  if and only if $M_1 \cap M_2 = 0$;
\item $H_i$ is a rational subgroup of $Tr_1(r, \Q)$ if and only if $M_i$ is a Lie subalgebra
of $Tr_0(r, \Q)$.
\end{enumerate}
\end{prop}
\begin{proof}
The first claim follows from Lemma \ref{logExpComm}.

The second claim follows easily from the fact that $\log$ and $\exp$ are inverse bijections. 
Suppose that $H_1 \cap H_2 = 1$, and let $m \in M_1 \cap M_2$. 
Then there exist $h_i \in H_i$ such that $m = \log(h_1) = \log(h_2)$.
Thus $\exp(m) = h_1 = h_2 = 1$ and hence $m = \log(1) = 0$.
Conversely suppose that $M_1 \cap M_2 = 0$ and let $h \in H_1 \cap H_2$.
Then $\log(h) \in M_1 \cap M_2 = 0$, so $\log(h) =0$.  Therefore $h = 1$. 

We now prove the third claim. 
If $H_1$ is a rational subgroup, then $M_1$ is a Lie subalgebra (see Theorem 2 on page 104 of \cite{Segal:book}). 
For the converse, suppose that $M_1$ is a Lie subalgebra.
There is an operator $\star$ on $Tr_0(r, \Q)$ defined using the Baker-Campbell-Hausdorff formula:
for $u,v \in Tr_0(r, \Q)$, $u \star v = u + v + l$, where $l$ is a certain $\Q$-linear combination of 
repeated Lie brackets of $u$ and $v$. This $\star$ operator satisfies 
$\exp({u \star v}) = \exp{u} \exp{v}$ for all $u,v \in Tr_0(r, \Q)$. 
(For a definition and properties of $\star$, see p. 102 in \cite{Segal:book}.)
Since $M_1$ is closed under $\star$, it follows that $H_1$ is closed under multiplication. 
Since $\exp(qu) = (\exp{u})^q$ for all $u \in Tr_0(r, \Q)$ and all $q \in \Q$, 
$H_1$ is closed under the taking of roots and inverses.
This establishes the third claim in our proposition. 
\end{proof}

We are now in a position to state our desired result concerning the uniqueness of decompositions
in the rational closure of a finitely generated torsion-free nilpotent group.

\begin{prop}
\label{rationalUnique}
Let $G$ be a finitely generated torsion-free nilpotent group. 
Suppose that we have two decompositions of $\overline{G}$ as the direct product of 
nontrivial rational subgroups which are themselves rationally indecomposable:
\begin{eqnarray*}
\overline{G} & = & R_1 \times R_2 \times \cdots \times R_m \\
& = & K_1 \times K_2 \times \cdots \times K_n.
\end{eqnarray*}
Let $\alpha_i$ be the projection of $\overline{G}$ onto $R_i$,
and let $\beta_i$ be the projection of $\overline{G}$ onto $K_i$.
Then $m = n$.  Furthermore, there is a way to reorder the factors such that
the following three properties hold for all $1 \leq i \leq m$:
\begin{enumerate}
\item $\alpha_i(K_i) = R_i$, and the restriction of $\alpha_i$ to $K_i$ is an isomorphism from 
$K_i$ to $R_i$ whose inverse is the restriction of $\beta_i$ to $R_i$, and
\item 
\begin{displaymath}
\overline{G}  = R_1 \times R_{2}  \times \cdots \times R_{i-1} \times K_i \times R_{i+1} \times \cdots \times R_m.
\end{displaymath}
\end{enumerate}
\end{prop}
\begin{proof}
Let $L = \log(\overline{G})$.
Since $L$ is a Lie subalgebra of the finite dimensional Lie algebra $Tr_0(r, \Q)$, 
it is itself finite dimensional,  and hence it is both Artinian and Noetherian. 
Let $M_i = \log(R_i)$ and let $N_i = \log(K_i)$.
By Proposition \ref{logExpProps}, we have 
\begin{eqnarray*}
L & = & M_1 \oplus M_2 \oplus \cdots \oplus M_r \\
& = & N_1 \oplus N_2 \oplus \cdots \oplus N_s,
\end{eqnarray*}
and the conclusions of Theorem \ref{uniqueLie} and its corollary hold.
Applying the $\exp$ map to 
\begin{eqnarray*}
L & = & M_1 \oplus M_2 \oplus \cdots \oplus M_{i-1} \oplus N_{i} \oplus M_{i+1} \cdots \oplus M_m,
\end{eqnarray*}
we get 
\begin{displaymath}
\overline{G}  = R_1 \times R_{2}  \times \cdots \times R_{i-1} \times K_i \times R_{i+1} \times \cdots \times R_m.
\end{displaymath}

Since $\alpha_i$ and $\beta_i$ are projections, they are clearly group homomorphisms. 
It is easy to see that $\alpha_i = \exp \circ \pi_i \circ \log$ and 
that $\beta_i = \exp \circ \psi_i \circ \log$:
let $h \in \overline{G}$, and let $r_i \in R_i$ such that $h = r_1 r_2 \cdots r_m$;
since the $r_i$'s commute,  
\begin{eqnarray*}
\exp \circ \pi_i \circ \log (h) & = & \exp \circ \pi_i (\log(r_1) + \cdots + \log(r_m)) \\
& = & \exp(\log(r_i)) = r_i = \alpha_i(h).
\end{eqnarray*}
It now follows from Theorem \ref{uniqueLie} 
that the suitable restrictions of $\alpha_i$ and $\beta_i$ are inverse bijections as desired. 
\end{proof}

An automorphism $\theta$ of a group $G$ is called {\em normal} if for all $x, y \in G$, $\theta(x^y) = (\theta(x))^y$. 
\begin{cor}
Let $G$ be a group satisfying the hypotheses of Proposition \ref{rationalUnique}.
If $\Theta_i$ is given by
\begin{displaymath}
\Theta_i((r_1, r_2, \ldots, r_m)) = (r_1, r_2, \ldots, r_{i-1}, \beta_i(r_i), r_{i+1}, \ldots, r_m)),
\end{displaymath}
then $\Theta_i$ is a normal automorphism of $\overline{G}$.
\end{cor}
\begin{proof}
The fact that $\Theta_i$ is an automorphism follows immediately from Proposition \ref{rationalUnique}.
To show that $\Theta_i$ is normal, it suffices to show that for all $r, s \in R_i$, $\beta_i(r^s) = (\beta_i(r))^s$, but
it is easy to see that $(\beta_i(r))^s = (\beta_i(r))^{\beta_i(s)} = \beta_i(r^s)$.
\end{proof}

We are interested in normal automorphisms because they fix centralizers:
\begin{rem} If $\theta$ is a normal automorphism of group $G$, and if $C_G(h)$ is the centralizer of $h$ in $G$, 
then $C_G(h) = \theta(C_G(h))$.
\end{rem}
\begin{proof}
Let $x \in C_G(h)$. Let $y = \theta^{-1}(x)$. Then 
$$\theta(h^y) =\theta(h)^{\theta(y)}=  \theta(h)^x = \theta(h^x) = \theta(h).$$
Therefore, $h^y = h$ and $y \in C_G(h)$. Thus $C_G(h) \subseteq \theta(C_G(h))$.
We obtain the opposite inclusion by considering the inverse of $\theta$. 
\end{proof}

Finally we will need to use the fact that there exist algorithms to determine whether 
$\overline{G}$ is rationally decomposable, and, if so, to compute a decomposition
(see Section 1.15 of \cite{deGraaf:book}, for example).
\begin{prop} 
Let $G$ be a finitely generated torsion-free nilpotent group.
There exists an algorithm 
to compute finite sets 
$A_1, A_2, \ldots, A_m$ 
of elements of $\overline{G}$
such that if $S_i$ is the smallest rational subgroup of $\overline{G}$ containing
$A_i$,
then 
\begin{displaymath}
\overline{G} = S_1 \times S_2 \times \cdots \times S_m.
\end{displaymath}
\end{prop}

\section{Abelian direct factors}
\label{abelianSection}

In this section we describe an algorithm for deciding whether a given 
finitely generated torsion-free nilpotent group has a nontrivial abelian direct factor.
In \cite{Baumslag:nilpDecomp}, Baumslag proves that factorizations of finitely generated
torsion-free nilpotent groups are not unique. In Section \ref{examples}
we will illustrate how the algorithms of this section provide
an easy proof of this fact. 

Let $G$ be a finitely generated torsion-free nilpotent group.  Suppose that $G$ has
an abelian factor.  Then $G \isom H\times \Z = H \times \<c\mid \ \>$.  Computing 
abelianizations (factor derived groups) 
we have $G/[G,G] \isom H/[H,H] \times  \<c\mid \ \>$ so that $c$
is a {\em primitive element} in  $G/[G,G]$.  Here $G/[G,G]$ can have torsion,
so by primitive element we mean that its image is part of a basis modulo the torsion subgroup.
Note that $c\in Z(G)$.  Here is a test for the presence of an abelian direct factor.

\begin{lem}  
\label{primitiveElement}
Let $G$ be a finitely generated torsion-free nilpotent group. 
Then $G$ has a nontrivial abelian direct factor if and only if  the image of $Z(G)$
in the factor derived group $G/[G,G]$ contains a primitive element.
\end{lem}

\proof  We know from above that the condition is necessary.  For sufficiency, suppose we
have an element $c\in Z(G)$ which is primitive.  Then there is a retraction
$\theta: G \onto \<c\mid \ \>$.  Since  $c\in Z(G)$, this  gives a direct product decomposition
$G = \ker \theta \times   \<c\mid \ \>$. \qed

\begin{thm}\label{thm:a}  There is an algorithm to determine of an arbitrary finite presentation of 
a torsion-free nilpotent group $G$ whether or not $G$ has an abelian direct factor.
If so, the algorithm expresses $G$ as $G \isom G_1 \times \Z^r$ where $G_1$
has no nontrivial abelian direct factor.
\end{thm}
\proof Let $W$ be the free abelian group $G/T$, where $T$ is the pullback in $G$ of the 
torsion subgroup of $G/[G,G]$, and let $n$ be the rank of $W$.
Let $V$ be the subgroup of $W$ given by $V = Z(G) [G,G] / [G,G]$,
and let $k$ be the rank of $V$. 
We can compute a basis for $W$ and a set of generators for $V$.
We can use a Smith normal form calculation to determine if $V$ contains a primitive element 
of $W$ as follows. Let $M$ be the $n \times k$ matrix whose $j$'th column is the $j$'th 
generator for $V$, expressed in terms of our basis for $W$.
Compute $P \in \mbox{Gl}_n(\Z)$ and $Q \in \mbox{Gl}_k(\Z)$ such that $PMQ = S$,
where $S$ is in Smith normal form. 
Then $V$ contains a primitive element of $W$ if and only if $S_{1,1} = 1$, 
in which case the first column of $P^{-1}$ is a primitive element of $W$ 
which is also an element of $V$. 
\qed

\section{Nonabelian direct factors}
\label{nonabelianSection}

In the previous section we described an algorithm for deciding if $G$ has a nontrivial abelian direct factor;
in this section we will develop an algorithm for deciding if $G$ has a decomposition as the the direct product 
of indecomposable nonabelian factors. In order to do so, we prove some structural theorems about the relationship between
the decompositions of a finitely generated torsion-free nilpotent group and the decompositions of 
its rational closure.

We begin with two definitions which will help simplify the exposition that follows. 
\begin{defn}
We say that decomposition $\overline{G} = R_1 \times R_2$  {\em matches} decomposition $G = G_1 \times G_2$   if 
$\overline{G_i} Z(\overline{G}) = R_i Z(\overline{G})$. 
\end{defn}

\begin{defn}
We say that decomposition $\overline{G} = R_1 \times R_2$  {\em gives rise to a decomposition of $G/Z(G)$}
if when $X_i = R_i Z(\overline{G}) \cap G$,
then $G = X_1 X_2$, $[X_1, X_2] =1$ and $X_1 \cap X_2 = Z(G)$. 
\end{defn}

We begin with some technical results.
In Proposition \ref{rationalUnique} we saw that decompositions of $\overline{G}$ 
are unique up to isomorphism. We will need a result which is slightly stronger in a way. 
\begin{prop}
\label{first}
Let $G$ be a finitely generated torsion-free nilpotent group. 
Suppose that we have two decompositions of $\overline{G}$ as the direct product of 
nontrivial rational subgroups which are themselves rationally indecomposable:
\begin{eqnarray*}
\overline{G} & = & R_1 \times R_2 \times \cdots \times R_m \\
& = & K_1 \times K_2 \times \cdots \times K_m
\end{eqnarray*}
and that the $K_i$'s have been permuted so that the conclusions 
of Proposition \ref{rationalUnique} hold. 
Then for all $i$,
$K_iZ(\overline{G}) = R_iZ(\overline{G})$.
\end{prop}
\begin{proof}
It suffices to prove the proposition for $i=1$. 
Let $R = R_2 \times R_3 \times \cdots \times R_m$. 
Let $\Theta$ be the normal automorphism $\Theta_1$ whose existence is posited in Proposition
\ref{rationalUnique}, so $\Theta(R_1) = K_1$ and $\Theta$ fixes every element of $R$. 
Fix $h \in R$.
The centralizer $C_{\overline{G}}(h) = R_1 \times C_{R}(h)$.
$\Theta(R_1 \times C_{R}(h)) = K_1 \times C_{R}(h)$. 
Since $\Theta$ fixes centralizers, we see that $R_1 \times C_{R}(h) = K_1 \times C_{R}(h)$,
and hence that $K_1 \leq R_1 \times C_{R}(h)$.
Since this holds for all $h \in R$, we get
$$ K_1 \leq \cap_{h \in R} (R_1 \times C_{R}(h)) = R_1 \times \cap_{h \in R} C_{R}(h)
= R_1 \times Z(R),$$
so $K_1 \leq R_1 Z(\overline{G})$. 

By considering the inverse of $\Theta$, we see that
$R_1 \leq K_1 Z(\overline{G})$, so our result now follows. 
\end{proof}

The following corollary establishes that if $G = G_1 \times G_2$
we can use a decomposition of $\overline{G}$ to make a finite list 
of decompositions of $\overline{G}$ as $\overline{G} = R_1 \times R_2$
and trust that at least one of those decompositions matches the 
decomposition $G = G_1 \times G_2$.
\begin{cor} 
\label{inclusionsCor}
Let $G$ be a finitely generated torsion-free nilpotent group. 
Suppose that $\overline{G} = S_1 \times S_2 \times \cdots \times S_m$ is a decompostion
of $\overline{G}$ into nontrivial rational subgroups each of which is 
rationally indecomposable.
Suppose furthermore that $G = G_1 \times G_2$ is a nontrivial decomposition of $G$.
Then it is possible to reorder the $S_i$'s and to choose $j$ such that 
if $R_1 = S_1 \times S_2 \times \cdots \times S_j$ and $R_2 = S_{j+1} \times \cdots \times S_m$,
then the decomposition $\overline{G} = R_1 \times R_2$ matches the 
decomposition $G = G_1 \times G_2$.
\end{cor}
\begin{proof}
We can decompose $\overline{G_1}$ and $\overline{G_2}$ into nontrivial rationally indecomposable subgroups
as follows:
\begin{eqnarray*}
\overline{G_1} & = & T_1 \times \cdots \times T_j, \\
\overline{G_2} &  = & T_{j+1} \times \cdots \times T_m.
\end{eqnarray*}
By Proposition \ref{first} we can reorder the $S_i$'s such that for all $i$, 
$S_i Z(\overline{G}) = T_i Z(\overline{G})$.
Let $R_1 = S_1 \times S_2 \times \cdots S_j$ and $R_2 = S_{j+1} \times \cdots \times S_m$.
Then for $i=1,2$, $R_i Z(\overline{G}) = \overline{G_i} Z(\overline{G})$.
\end{proof}

We will rely repeatedly on the following obvious fact about the center of a direct product. 
\begin{lemma}
\label{centerDecomp}
Let $H$ be any group. 
Suppose that $H = H_1 \times H_2$. 
Then $Z(H) = Z(H_1) \times Z(H_2)$.
\end{lemma}
\begin{proof}
Let $z \in Z(G)$. Let $g_i \in H_i$ such that $z = g_1 g_2$. 
Let $h_1 \in H_1$. Then $z h_1 = (g_1 g_2) h_1 = (g_1 h_1) g_2$.
On the other hand, $h_1 z = (h_1 g_1) g_2$.
Therefore $h_1 g_1 = g_1 h_1$. Hence $g_1 \in Z(H_1)$. 
Likewise $g_2 \in Z(H_2)$. Hence $Z(G) = Z(H_1) Z(H_2) = Z(H_1) \times Z(H_2)$. 
\end{proof}

\begin{prop}
\label{XProp}
Let $G$ be a finitely generated torsion-free nilpotent group. 
Suppose that decomposition $\overline{G} = R_1 \times R_2$ matches decomposition $G = G_1 \times G_2$. 
Let $X_i = R_i Z(\overline{G}) \cap G$. 
Then $X_i = G_i Z(G)$.
\end{prop}
\begin{proof}
Let $x \in X_1$. Since $R_i Z(\overline{G}) = \overline{G_i}  Z(\overline{G})$, 
there exist $r_1 \in \overline{G_1}$ and $z \in Z(\overline{G})$ such that
$x = rz$. 
By Proposition \ref{easyDirection}, $\overline{G} = \overline{G_1} \times \overline{G_2}$ and hence by 
Lemma \ref{upperCentral}, 
$Z(\overline{G}) = Z(\overline{G_1}) \times Z(\overline{G_2})$. 
Therefore, there exist $z_i \in Z(\overline{G_i})$ such that $z = z_1 z_2$. 
Since $x \in G$, there exist $g_i \in G_i$ such that $x = g_1 g_2$.
Thus we have two ways of expressing $x$ as an element of $\overline{G_1} \times \overline{G_2}$:
$x = (r_1 z_1) z_2 = g_1 g_2$. Therefore, $r_1 z_1 = g_1$ and $z_2 = g_2$. 
Therefore, $r_1 z_1 \in G_1$ and $z_2 \in Z(G)$, so $x \in G_1 Z(G)$.
We have shown that $X_1 \subseteq G_1 Z(G)$. 
The other inclusion is clear since by Theorem \ref{upperCentral}
$Z(G) \subseteq Z(\overline{G})$. 
\end{proof}

\begin{cor}
\label{XCor}
Let $G$ be a finitely generated torsion-free nilpotent group. 
Suppose that decomposition $\overline{G} = R_1 \times R_2$ matches decomposition
$G = G_1 \times G_2$.
Then $\overline{G} = R_1 \times R_2$ gives rise to a decomposition of $G/Z(G)$. 
\end{cor}
\begin{proof}
By Proposition \ref{XProp}, $X_i = G_i Z(G)$. 
It is obvious that $G = X_1 X_2$.
To see that $X_1 \cap X_2 = Z(G)$, it suffices to show that $G_1 \cap G_2 Z(G) \leq Z(G)$.
Suppose that $g_1 \in G_1 \cap G_2 Z(G)$. 
By Lemma \ref{centerDecomp}, $g_1 = g_2 z_1 z_2$ for some $g_2 \in G_2$
and $z_i \in Z(G_i)$. Therefore $g_1 = z_1 (g_2 z_2)$. Hence $g_1 = z_1 \in Z(G_1)$.
Hence $g_1 \in Z(G)$. 
$$[X_1, X_2] = [G_1Z(G), G_2 Z(G)] = [G_1, G_2] = 1.$$
\end{proof}

\begin{lemma}
\label{little2}
Let $G$ be a finitely generated torsion-free nilpotent group. 
Suppose that decomposition $\overline{G} = R_1 \times R_2$ matches decomposition
$G = G_1 \times G_2$.
Let $X_i = R_i Z(\overline{G}) \cap G$.
If $H_1$ and $H_2$ are subgroups of $G$ such that 
$X_i = H_i Z(G)$ and 
$Z(G) = Z(H_1) \times Z(H_2)$, 
then $G = H_1 \times H_2$.
\end{lemma}
\begin{proof}
By Proposition \ref{XProp}, $X_i = G_i Z(G)$ and by Corollary \ref{XCor},
 $G = X_1 X_2$, $[X_1, X_2] =1$ and $X_1 \cap X_2 = Z(G)$.

We first show that $G = H_1 H_2$. 
Let $g \in G$. Then $g = x_1 x_2$ for some $x_i \in X_i$. 
Since each $x_i \in H_i Z(G)$, 
$g = h_1 h_2 z$ for some $h_i \in H_i$ and $z\in Z(G)$. 
By assumption
there exist $z_i \in Z(G) \cap H_i$ such that $z = z_1 z_2$.
Therefore $g = h_1 h_2 z_1 z_2 =  (h_1 z_1) (h_2 z_2) \in H_1 H_2$. 

We next show that $H_1 \cap H_2 = 1$. 
Suppose that $g \in H_1 \cap H_2$. 
Then $g \in X_1 \cap X_2 = Z(G)$, and 
hence $g \in H_1 \cap Z(G)$. Since $Z(G) = Z(H_1) \times Z(H_2)$, $g \in Z(H_1)$.
Similarly $g \in Z(H_2)$.
Since $Z(H_1) \cap Z(H_2) = 1$, $g = 1$.

Finally, $$[H_1, H_2] \leq [X_1, X_2] = [G_1 Z(G), G_2 Z(G)= [G_1, G_2] = 1.$$
\end{proof}

\begin{lemma}
\label{little3}
Let $G$ be a finitely generated torsion-free nilpotent group. 
Suppose that the decomposition $\overline{G} = R_1 \times R_2$ matches decomposition
$G = G_1 \times G_2$.
Let $X_i = R_i Z(\overline{G}) \cap G$. 
Suppose that $a_1, a_2, \ldots, a_k$ are elements of $X_1$ such that 
$$a_1 Z(G), a_2 Z(G), \ldots, a_k Z(G)$$ is a consistent polycyclic generating sequence
for $X_1/Z(G)$. 
Suppose that $b_1, b_2, \ldots, b_l$ is a corresponding sequence of elements of $X_2$.
Define $\widetilde{G_1}= \langle a_1, a_2, \ldots, a_k, Z(G_1) \rangle$
and $\widetilde{G_2} = \langle b_1, b_2, \ldots, b_l, Z(G_2) \rangle$.
Then $G = \widetilde{G_1} \times \widetilde{G_2}$. 
\end{lemma}
\begin{proof}
We first notice that by Corollary \ref{XCor}, the decomposition $\overline{G} = R_1 \times R_2$
gives rise to a decomposition of $G/Z(G)$ and so by Proposition \ref{XProp},
$G_i Z(G) = X_i$.
Also notice that $\widetilde{G_i} Z(G) = G_i Z(G) = X_i$. 
By Lemma \ref{little2} it suffices to show that $Z(\widetilde{G_i}) = Z(G_i)$. 
Let $z \in Z(\widetilde{G_i})$. There exist integers $\alpha_i$ and an 
element $z_1 \in Z(G_1)$ such that
$$z = a_1^{\alpha_1} a_2^{\alpha_2} \cdots a_k^{\alpha_k} z_1.$$
This implies that $a_1^{\alpha_1} a_2^{\alpha_2} \cdots a_k^{\alpha_k} \in Z(G)$. 
But by our choice of the $a_i$'s, this in turn implies that each $\alpha_i$ is equal to $0$.
Therefore, $z = z_1$ and hence $z \in Z(G_1)$. The opposite inclusion is obvious.
\end{proof}

We now describe an algorithm to test whether a given decomposition $\overline{G} = R_1 \times R_2$
which gives rise to a decomposition of $G/Z(G)$ satisfies the further property that it
matches a decomposition 
$G = G_1 \times G_2$. If it does, then the algorithm produces a decomposition. 

\begin{thm}
\label{testingX1X2}
Let $G$ be a finitely generated torsion-free nilpotent group. 
Suppose that the decomposition $\overline{G} = R_1 \times R_2$ gives rise
to a decomposition of $G/Z(G)$. Let $X_i = R_i Z(\overline{G}) \cap G$. 
Suppose that $a_1, a_2, \ldots, a_k$ are elements of $X_1$ such that 
$$a_1 Z(G), a_2 Z(G), \ldots, a_k Z(G)$$ is a consistent polycyclic generating sequence
for $X_1/Z(G)$. 
Suppose that $b_1, b_2, \ldots, b_l$ is a corresponding sequence of elements of $X_2$.
Let $H_1 = \langle a_1, a_2, \ldots, a_k \rangle$ and
let $H_2 = \langle b_1, b_2, \ldots, b_l \rangle$. 
If the decomposition $\overline{G} = R_1 \times R_2$ matches 
a decomposition $G = G_1 \times G_2$, then 
there is a decomposition  $Z(G) = Z_1 \times Z_2$ such that 
$Z(H_1) \leq Z_1$ and $Z(H_2) \leq Z_2$.
Conversely, if
there is a decomposition $Z(G) = Z_1 \times Z_2$ such that 
$Z(H_1) \leq Z_1$ and $Z(H_2) \leq Z_2$,
then $G = G_1 \times G_2$ where $G_i = H_i Z_i$,
and $\overline{G} = R_1 \times R_2$ matches this decomposition. 
\end{thm}
\begin{proof}
Suppose that there is a decomposition $G = G_1 \times G_2$
that matches the decomposition $\overline{G} = R_1 \times R_2$.
To show that there is a decomposition  $Z(G) = Z_1 \times Z_2$ such that 
$Z(H_1) \leq Z_1$ and $Z(H_2) \leq Z_2$, we let 
$\widetilde{G_1} = \langle a_1, a_2, \ldots, a_k, Z(G_1) \rangle$
and $\widetilde{G_2} = \langle b_1, b_2, \ldots, b_l, Z(G_2) \rangle$.
Then by Lemma \ref{little3}, 
$G = \widetilde{G_1} \times \widetilde{G_2}$.
Let $Z_i = Z(\widetilde{G_i})$. 
Then $Z(G) = Z_1 \times Z_2$. 
We will show that $Z(H_i) \leq Z_i$. 
Let $z \in Z(H_i)$. Then $z \in \widetilde{G_i}$ and $z$ commutes with everything in $\widetilde{G_i}$. 
Therefore, $z \in Z_i$.

Suppose that 
there is a decomposition $Z(G) = Z_1 \times Z_2$ such that 
$Z(H_1) \leq Z_1$ and $Z(H_2) \leq Z_2$, and we 
let $G_i = H_iZ_i$. Notice that $X_i = G_i Z(G)$.
Since we are assuming that our decomposition $\overline{G} = R_1 \times R_2$
gives rise to a decomposition of $G/Z(G)$,
we know that $G = X_1 X_2$, $[X_1, X_2] =1$ and $X_1 \cap X_2 = Z(G)$,
from which it follows that $[G_1, G_2] = [X_1, X_2] = 1$. 

We now show that $G= G_1 G_2$. 
Let $g \in G$. Since $g \in X_1 X_2$, $g = h_1 z_1 h_2 z_2$ for some 
$h_i \in H_i$ and $z_i \in Z(G)$. 
Since $z_1 z_2 \in Z(G)$ there exists $z_1' \in Z_1$
and $z_2' \in Z_2$ such that $z_1 z_2 = z_1' z_2'$. 
Now $g = h_1 z_1' h_2 z_2' \in G_1 G_2$.

Finally, $G_1 \cap G_2 = 1$. 
To see this, suppose that $g \in G_1 \cap G_2$.
Then $g \in X_1 \cap X_2 = Z(G)$ so $g = z_1 z_2$ for some 
$z_i \in Z_i$. 
Therefore, $z_1 \in G_2$, so $z_1 = h_2 z_2'$ for some 
$h_2 \in H_2$ and $z_2' \in Z_2$. 
Therefore $h_2 \in Z(G)$ and hence $h_2 \in Z_2$. 
Now $z_1 \in Z_1 \cap Z_2 = 1$, so $z_1 = 1$. 
Likewise $z_2 = 1$, so $g = 1$. 
We have proven that $G = G_1 \times G_2$.

It remains only to show that $\overline{G} = R_1 \times R_2$ 
matches $G = G_1 \times G_2$, i.e. that $\overline{G_i} Z(\overline{G}) = R_i Z(\overline{G})$.
Let $x \in \overline{G_i}$. There exists a positive integer $m$ such that $x^m \in G_i$,
so $x^m \in X_i$. Hence $x^m \in R_i Z(\overline{G})$. Since $R_i$ is rational
and by Theorem \ref{upperCentral} $Z(\overline{G}) = \overline{Z(G)}$,
$R_i Z(\overline{G})$ is rational and hence $x \in R_i Z(\overline{G})$.

Conversely, suppose that $x \in R_i$. There exists a positive integer $m$ such 
that $x^m \in G$. Thus $x^m \in R_i \cap G \leq X_i = G_i Z(G)$. 
Therefore $x \in \overline{G_i Z(G)} = \overline{G_i} Z(\overline{G})$. 
We have shown that the decomposition $\overline{G} = R_1 \times R_2$ matches 
the decomposition $G = G_1 \times G_2$. 
\end{proof}

Suppose that $A$ is a finitely generated abelian group, so it is a direct product 
of finitely many cyclic groups.   We recall that a subgroup $V$ of $A$ is {\em pure} if 
$w^k \in V$ implies $w\in V$.   A pure subgroup of $A$ is a direct factor.  Also the
intersection of pure subgroups is pure.  We need the following variation on the Smith
normal form algorithm.

\begin{lem}
\label{purify}  There is an algorithm which, given 
a finitely generated free abelian group $A$ and a finitely gnerated subgroup  $V$,
finds a set of generators for the  smallest subgroup $W \geq V$ which is pure in $A$. 
The algorithm determines whether $V$ is already pure, and finds a 
splitting $A= W\times A_2$.   Also, $\rank_\Z W = \rank_\Z V$. 
\end{lem}

\begin{proof}  Since we are dealing with abelian groups, it is convenient to use
additive notation.  Let $W$ be the set of all elements $w \in A$ such that $kw \in V$ for some $k \in \Z$. 
Then $W$ is the smallest pure subgroup of $A$ containing $V$.  The usual Smith normal form 
computations giving the structure of finitely generated abelian groups can be applied here.  So if $M$
is the integer matrix expressing the generators of $V$ in terms of the basis for $A$, the computation
does row and column operations to obtain a new matrix $PMQ$ in canonical form with integer invariants
$c_1,\ldots c_m$ where $1\leq c_i\in\bZ$ and $c_i | c_{i+1}$.  Here $m$ is the rank of $V$ and 
$n$ is the rank of $A$.   Using the matrix $P^{-1}$ we can
find the new basis $\{w_1,\ldots,w_n\}$ such that $V$  is generated by $\{c_1w_1,\ldots,c_mw_m\}$.
Clearly  $W = gp\{w_1,\ldots,w_m\}$ is the smallest pure subgroup containing $V$.  Notice that
$V$ itself is pure if and only if all the $c_i=1$.  Also, if we put $A_2 = gp\{w_{m+1},\ldots,w_n\}$
then we have a direct product splitting $A= W\times A_2$.
\end{proof}

The following lemma shows that the condition of Theorem \ref{testingX1X2}
is easily testable. 
\begin{lem}
\label{testingAbelSplit}  There is an algorithm which, given 
a finitely generated free abelian group $A$ and two finitely gnerated subgroups  $V_1$ and $V_2$,
determines whether or not there exists a splitting $A = A_1 \times A_2$ such that $A_i \geq V_i$.  If such 
splitting exists, the algorithm produces such a  decomposition explicitly.
\end{lem}

\begin{proof}  Given the subgroups $V_i$,  we use the algorithm of Lemma \ref{purify} 
to compute the smallest pure subgroups
$W_i\geq V_i$.  Notice that $W_1\cap W_2 = \{0\}$ if and only if 
$V_1\cap V_2 = \{0\}$. 

Now consider the subgroup $B=gp\{W_1,W_2\}$.  
In general $B$ may not be pure in $A$.
But if there is a splitting $A = A_1 \times A_2$ such that $A_i \geq V_i$, then $A_i\geq W_i\geq V_i$, 
and hence $B$ is pure in $A$ and $\rank_\Z  B = \rank_\Z  W_1 + \rank_\Z  W_2$.  Conversely
suppose that $B$ is pure in $A$.
Then rank formula  is $\rank_\Z  B = \rank_\Z  W_1 + \rank_\Z  W_2 - \rank_\Z(W_1\cap W_2)$.
But $\rank_\Z  B = \rank_\Z  W_1 + \rank_\Z  W_2 $ implies   $W_1\cap W_2 = \{0\}$  so that $B= W_1\times W_2$.
Hence we can find the desired splitting.  Since we can test purity and compute ranks, we can determine
whether such a splitting of $A$ exists, and if so explicitly find one.
\end{proof}

Notice that in the special case when one of the rational factors is abelian,
Theorem \ref{testingX1X2} is vacuous. In this case $X_1 = G$, $X_2 = Z(G)$
and $H_2 = 1$. If we let $Z_1 = Z(G)$ and $Z_2 = 1$ we get a splitting
$Z(G) = Z_1 \times Z_2$ with $Z(H_1) \leq Z_1$ and $Z(H_2) = 1$. 
We then find that $G_1 = G$ and $G_2 = 1$, so we have proven the 
existence of the trivial decomposition for $G$, 
that is, we have proven nothing. 
However, if $G$ has a decomposition $G = G_1 \times G_2$ where neither
factor is abelian, 
then given any decomposition of $\overline{G}$ as the direct product
of rationally indecomposable groups, 
there will be a way to group the indecomposable factors to obtain
$\overline{G} = R_1 \times R_2$ such that neither $R_i$ is abelian. 
Thus, the theorems described in this section provide a test for the existence of a nontrivial
nonabelian direct factor. 
In Section \ref{abelianSection} we described a separate algorithm
for deciding if $G$ has a nontrivial abelian factor.

We can now summarize our algorithm for deciding whether or not there exist nonabelian groups 
$G_1$ and $G_2$ such that $G = G_1 \times G_2$. 
We compute a decomposition  $\overline{G} = S_1 \times S_2 \times \cdots \times S_m$ 
of $\overline{G}$ into nontrivial rational subgroups each of which is 
rationally indecomposable.
There are a finite number of ways to group the $S_i$'s to obtain 
$\overline{G} = R_1 \times R_2$ where neither $R_i$ is abelian.
We further restrict our attention to decompositions that give
rise to a decomposition of $G/Z(G)$ which we determine by
computing $X_i = R_i Z(\overline{G}) \cap G$ 
and deciding whether 
$G = X_1 X_2$, $[X_1, X_2] = 1$ and $X_1 \cap X_2 = Z(G)$.
If none of our computed decompositions $\overline{G} = R_1 \times R_2$ 
give rise to a decomposition of $G/Z(G)$, we conclude that $G$ does not 
have a decomposition into two nonabelian factors. 

Otherwise, 
for each of the computed decompositions of $\overline{G}$ that does give rise to a
decomposition of $G/Z(G)$,  we decide
whether that decomposition matches a decomposition $G = G_1 \times G_2$ as follows.
We define $H_i$ using a consistent polycyclic generating sequence for $X_i/Z(G)$ as in Theorem \ref{testingX1X2}. 
We compute $Z(H_1)$ and $Z(H_2)$. 
By Theorem \ref{testingX1X2} and Remark \ref{testingAbelSplit},
there exists a decomposition of $G$ corresponding with $\overline{G} = R_1 \times R_2$ if and only 
if $Z(H_1) \cap Z(H_2) = 1$. 
In this case we compute a decomposition $Z(G) = Z_1 \times Z_2$ with 
$Z(H_i) \leq Z_i$, and we let $G_i = H_i Z_i$. 
Then $\overline{G} = R_1 \times R_2$
matches the decomposition $G = G_1 \times G_2$. 
If, after consideration of all our computed decompositions $\overline{G} = R_1 \times R_2$ 
that give rise to decompositions of $G/Z(G)$,
we find that none matches a decomposition of $G$,
then we deduce that $G$ is cannot be decomposed as the direct
product of two nonabelian factors. 

We have now completed the proof of our main theorem. 
\begin{thm} 
\label{thm:b}
There is an algorithm to determine of an arbitrary finite presentation of 
a torsion-free nilpotent group $G$ without abelian direct factors, 
whether or not $G$ has a nontrivial direct decomposition.
If so, the algorithm expresses $G$ as $G \isom G_1 \times\ldots  \times G_n$ where 
each $G_i$ is directly indecomposable.
\end{thm}

\section{Examples}
\label{examples}

In this section we use the examples from \cite{Baumslag:nilpDecomp} to illustrate how 
our algorithms work.
In doing so we also provide an easy proof of the theorem
in \cite{Baumslag:nilpDecomp} asserting that 
direct product decompositions of finitely generated torsion-free groups may not be unique. 

We begin by describing some examples of torsion free nilpotent groups 
which we shall denote by  $G_p$  for $p>1$ (and which are denoted $G(1,p)$
in \cite{Baumslag:nilpDecomp}). 
Let $A = \langle a,b,c \rangle$ be the free abelian group of rank 3 on the listed generators.
Then the HNN-extension 
$$B= \< A, t \mid  a^t = a b,  b^t = b c,  c^t = c \>$$
is torsion-free and nilpotent of class 3.  
Let  $F = \langle f \rangle$ be the free abelian group of rank 1 on the given generator, and put  
$K= B\times F$.   We define a subgroup $K \subset G_p  \subset \ov{K}$ by
$G_p = \langle K, s \rangle$ where $s^p = bf$.  Thus $s$ is the unique $p$-th root of $bf$ in 
$\ov{K}$.

In Lemma 3 of \cite{Baumslag:nilpDecomp} Baumslag proves that $G_p$ is not directly
decomposable. 
Here we provide a simpler proof of this lemma using our algorithm.
We begin with some simple observations about the structure of $G_p$.
Notice that neither $b^{\frac{1}{p}}$ nor $f^{\frac{1}{p}}$ is an element of $G_p$,
and yet $c^{\frac{1}{p}}$ is an element of $G_p$: an easy calculation shows
that $[t, s^{-1}]^p = c$. 
The center of $G_p$ is given by $Z(G_p) = \langle c^{\frac{1}{p}}, f\rangle.$
The derived subgroup of $G_p$ is given by $[G_p, G_p] = \langle b, c^{\frac{1}{p}} \rangle$. 
The  abelianization of $G_p$ is free abelian on $\{t,a,s\}$.

We use the algorithm of Section \ref{abelianSection} to show that $G_p$ has 
no abelian direct factor. 
The image of $Z(G_p)$ in the abelianization is generated by the image of $f$ which is equal to the image of $s^p$.  Clearly the image of $s^p$
is not primitive in the abelianization. By Lemma \ref{primitiveElement}, $G_p$ has no abelian direct factor.

We use Proposition \ref{easyDirection}
to show that $G_p$ cannot be decomposed as the direct product of nonabelian factors
by showing that $\overline{G_p}$ cannot be decomposed  as the direct product 
of nonabelian factors. 
The rational closure of $G_p$ has the following decomposition
into rationally indecomposable groups:  $\ov{G_p} = \ov{K} = \ov{B} \times \ov{F}$.
To see that $\overline{B}$ is rationally indecomposable, observe that 
the center of $B$ is the cyclic group generated
by $c$, and hence, by Theorem \ref{upperCentral}, 
the center of $\ov{B}$ is isomorphic to $\Q$.
Since every factor in a splitting has a nontrivial center, 
this shows that $\overline{B}$ is rationally indecomposable. 

Next we consider $D= G_p \times G_q$, where $p$ and $q$ are relatively prime. 
We will prove that this decomposition (as the direct product of indecomposable groups) is not unique by using the algorithm of Section \ref{abelianSection}
to find an abelian direct factor. 
We will name the generators of $G_q$ using the corresponding Greek letters 
$\alpha, \beta, \gamma, \sigma$ and $\phi$,
so, for example, $\sigma^q = \beta \phi$. 
The abelianization of $D$ is free abelian with basis 
$\{t,a,s,\tau,\alpha,\sigma\}$.
The image of $Z(D)$ in the abelianization is generated by the images of $s^p$ and $\sigma^q$.   
We can perform a Smith normal form calculation as described in Section \ref{abelianSection},
but in this case it is easy to see that 
if $l$ and $m$ are integers such that $lp+mq = 1$, 
then $\left( \begin{array}{cc} p & m \\ -q & l \end{array} \right)$
is invertible, and hence $b^{-1} s^p \sigma^{-q} \beta = f \phi^{-1}$ is a primitive element of the abelianization that is central in $D$.
Thus $f \phi^{-1}$ generates a cyclic direct factor $T$ of $D$.
We have proved that the decomposition of $D$ is not unique, even up to isomorphism. 

We use our algorithm to show that the complement $S$ to $T$ in $D$ is itself indecomposable. 
We are going consider $S$ as the quotient  of $D$ obtained by identifying $f$ and $\phi$ (that is $D/T$).
Notice that in $D$ the subgroup $T$ intersects each of $G_p$ and
$G_q$ trivially, and so the quotient $D/T = S$ is a direct product with
central amalgamation, and $S$ is isomorphic to the subgroup 
$\tilde{S}$ of 
$\overline{D}$ generated by $t, a, s, \tau, \alpha, f$ and $\beta^{1/q} f^{1/q}$. 
To simplify the notation for the rest of this section we will refer to $\tilde{S}$ as $S$, 
even though $\tilde{S}$ is not actually a subgroup of $D$, but rather it is a subgroup
of $\overline{D}$ that is isomorphic to a direct complement of $T$ in $D$. 
Notice that with this notation, the derived subgroup of $S$ is given by 
$[S,S] = \langle  b, c^{\frac{1}{p}}, \beta, \gamma^{\frac{1}{q}} \rangle$
and the center of $S$ is given by $\langle c^{\frac{1}{p}}, \gamma^{\frac{1}{q}}, f \rangle$. 

We first show that $S$ does not have an abelian factor. 
In the abelianization of $S$, the image of the center is generated by $f$,
which is also the image of $s^p$ and $\sigma^q$, and so is a $pq$'th power.
Therefore $f$ is not a primitive element of
the abelianization. Thus by Lemma \ref{primitiveElement}, $S$ does 
not have an abelian factor. 

Finally we will show that $S$ is not the direct product of two nonabelian factors. 
Note that $\overline{S}$ decomposes into rationally indecomposable factors as follows:
$\overline{S} = \overline{B_p} \times \overline{B_q} \times \overline{F}$, 
where we use $B_p$ to denote the subgroup of $S$ generated by $t, a$,
and $B_q$ to denote the subgroup of $S$ generated by $\tau, \alpha$,
so $B \cong B_p \cong B_q$. 
The first step of the algorithm demands that we
consider all ways of partitioning the given factors of 
$\overline{S}$ to obtain $\overline{S} = R_1 \times R_2$,
where both $R_i$'s are rational and nonabelian. 
There are essentially two partitions here to consider which are entirely analogous. 
So it suffices look at
the case when $R_1 = \overline{B_p}$
and $R_2 = \overline{B_q} \times \overline{F}$. 

We must first decide whether $\overline{S} = R_1 \times R_2$
gives rise to a decomposition of $S/Z(S)$.
Let $X_i = R_i Z(\overline{S}) \cap S$. We find that 
\begin{eqnarray*}
X_1 & = &  \langle t, a, b^{\frac{1}{p}} f^{\frac{1}{p}}, c^{\frac{1}{p}}, \gamma^{\frac{1}{q}}, f \rangle, \\
X_2 & = & \langle \tau, \alpha,\beta^{\frac{1}{q}} f^{\frac{1}{q}}, c^{\frac{1}{p}}, \gamma^{\frac{1}{q}}, f \rangle.
\end{eqnarray*}

We now define $H_1$ and $H_2$ according to the requirements of Theorem \ref{testingX1X2}.
The images of $t, a, b^{\frac{1}{p}} f^{\frac{1}{p}}$ form a consistent polycyclic generating sequence for $X_1/Z(S)$
and the images of $\tau, \alpha, \beta^{\frac{1}{q}} f^{\frac{1}{q}}$ form a consistent polycyclic generating sequence for $X_2/Z(S)$. 
We let
\begin{eqnarray*}
H_1 & = & \langle  t, a, b^{\frac{1}{p}} f^{\frac{1}{p}} \rangle, \\
H_2 & = & \langle \tau, \alpha, \beta^{\frac{1}{q}} f^{\frac{1}{q}} \rangle.
\end{eqnarray*}
Then $f \in Z(H_1) \cap Z(H_2)$, and so by Theorem \ref{testingX1X2} and Remark \ref{testingAbelSplit}, there is no decomposition $S = G_1 \times G_2$ that matches the decomposition $\overline{S} = R_1 \times R_2$.
We have completed our proof that $S$ is indecomposable. 

\bibliographystyle{plain}
\bibliography{gbcmgoBib}

\end{document}